\documentclass[11pt]{amsart}

\usepackage{amssymb,amsmath}
\usepackage{bbm}
\usepackage{a4wide}

\theoremstyle{plain}
\newtheorem{theorem}{Theorem}
\newtheorem{prop}[theorem]{Proposition}
\newtheorem{lemma}[theorem]{Lemma}
\newtheorem{coro}[theorem]{Corollary}

\theoremstyle{definition}
\newtheorem{remark}{Remark}

\newcommand{\ts}{\hspace{0.5pt}}
\newcommand{\nts}{\hspace{-0.5pt}}
\newcommand{\CC}{\mathbb{C}\ts}
\newcommand{\RR}{\mathbb{R}\ts}

\newcommand{\ZZ}{\mathbb{Z}}
\newcommand{\SSS}{\mathbb{S}}

\newcommand{\NN}{\mathbb{N}}
\newcommand{\XX}{\mathbb{X}}

\newcommand{\EE}{\mathbb{E}}
\newcommand{\eone}{\boldsymbol{e}^{}_{1}}
\newcommand{\etwo}{\boldsymbol{e}^{}_{2}}

\newcommand{\dd}{\, \mathrm{d}}

\begin{document}

\title[Pure Lebesgue diffraction spectrum]
{Planar dynamical systems with \\[2mm]
pure Lebesgue diffraction spectrum}

\author{Michael Baake}
\address{Fakult\"at f\"ur Mathematik, Universit\"at Bielefeld, \newline
\hspace*{\parindent}Postfach 100131, 33501 Bielefeld, Germany}
\email{mbaake@math.uni-bielefeld.de}

\author{Tom Ward}
\address{School of Mathematics, University of East Anglia, \newline
\hspace*{\parindent}Norwich NR4 7TJ, UK}
\email{t.ward@uea.ac.uk}

\begin{abstract}
  We examine the diffraction properties of lattice
dynamical systems of algebraic origin. It is well-known
that diverse dynamical properties occur
within this class. These include different orders of
mixing (or higher-order correlations), the presence
or absence of measure rigidity (restrictions on the
set of possible shift-invariant ergodic measures to
being those of algebraic origin), and different
entropy ranks (which may be viewed as the maximal
spatial dimension in which the system resembles an i.i.d.\
process). Despite these differences, it is shown that
the resulting diffraction spectra are essentially
indistinguishable, thus raising further difficulties
for the inverse problem of structure determination
from diffraction spectra. Some of them may be resolved
on the level of higher-order correlation functions,
which we also briefly compare.

\end{abstract}

\maketitle

\section{Introduction}

The diffraction measure is a characteristic attribute of a
translation bounded measure $\omega$ on Euclidean space (or on
any locally compact Abelian group). It emerges as the Fourier
transform $\widehat{\gamma}$ of the autocorrelation $\gamma$ of
$\omega$, and has important applications in crystallography,
because it describes the outcome of kinematic diffraction (from
X-rays or neutron scattering, say \cite{Cowley}). In recent
years, initiated by the discovery of quasicrystals (which are
non-periodic but nevertheless show pure Bragg diffraction), a
systematic study by many people has produced a reasonably
satisfactory understanding of the class of measures with a pure
point diffraction measure, meaning that $\widehat{\gamma}$ is a
pure point measure, without any continuous component; see
\cite{BLM,BM,Hof,Martin} and references therein for details.

Clearly, reality is more complicated than that, in the sense that real
world structures will (and do) show substantial continuous components
in addition to the point part in their diffraction
measure. Unfortunately, the methods around dynamical systems (in
particular, the relations between dynamical and diffraction spectra)
that are used to establish pure point spectra~\cite{BL,BLM,Martin} and
to explore some of the consequences~\cite{Lenz} do not seem to extend
to the treatment of systems with mixed spectrum~\cite{ME}, at least
not in sufficient generality.

More recently, systems with continuous spectral components have
been investigated also from a rigorous mathematical point of
view (see~\cite{BBM} for a survey and references), with a
number of unexpected results. In particular, the phenomenon of
homometry (meaning the existence of different measures with the
same autocorrelation) becomes more subtle. For example, on the
basis of the autocorrelation alone, it is not possible to
distinguish the Bernoulli comb (with completely positive
entropy) from the deterministic Rudin--Shapiro comb (with zero
entropy)~\cite{BG,HB}.

The purpose of this article is to show that, in higher
dimensions, new phenomena emerge. In particular, one can have
lattice systems of lower entropy rank that are again homometric
to the Bernoulli system (of full entropy rank).  Here, the
entropy rank means the largest dimension for which the marginal
of the invariant measure used has positive entropy.
Paradigmatic is the Ledrappier shift, which will be one of our
main examples, though we also discuss various extensions.  We
sum up with a sketch of a collection of systems with diverse
dynamical properties that all have pure Lebesgue diffraction
spectrum, which raises a number of interesting questions to
follow up.

\section{Ledrappier's shift}\label{sec:Ledrappier}

Let $x,y,z$ denote elements of $\ZZ^{2}$, and $u,v,w$ elements
of a subshift, which itself might be a group (for instance, $w$
will often refer to a generic element in the Ledrappier system
below). Let $\eone$ and $\etwo$ denote the standard horizontal
and vertical unit vectors in the Euclidean plane, and define
\begin{equation}\label{eq:def-ledrap}
   \XX_{\mathrm{L}} = \bigl\{ w= (w_{x})^{}_{x\in\ZZ^{2}}
      \in\{\pm1\}^{\ZZ^{2}} \mid w_{x} w_{x+\eone}
      w_{x+\etwo} = 1\mbox{ for all } x \in\mathbb Z^2 \bigr\} \ts .
\end{equation}
This defines a (closed) subshift, but also a closed Abelian
subgroup (under pointwise multiplication) of the compact group
$\{\pm 1 \}^{\ZZ^{2}}$. The topology that~$\XX_{\mathrm{L}}$
inherits is the product topology, in which two points (or
configurations) are close if they agree on a large
neighbourhood of~$0\in\ZZ^2$. It is written multiplicatively
here (in contrast to the additive notation in~\cite{L}) for
reasons that will become clear shortly. We equip
$\XX_{\mathrm{L}}$ with its Haar measure $\mu^{}_{\mathrm{L}}$,
and denote the $\ZZ^{2}$ shift action by $\alpha$, meaning that
\[
     \bigl( \alpha^{z}(w) \bigr)_{x}  =  \, w_{x+z} \ts ,
     \qquad \mbox{for all $x,z \in \ZZ^{2}$ and
     $w \in \XX_{\mathrm{L}}$.}
\]
This action is continuous (that is, for any~$z\in\ZZ^2$, the
map~$\alpha^{z}$ is a homeomorphism) since, for any
fixed~$z\in\ZZ^{2}$, nearby configurations are shifted to
nearby configurations by~$\alpha^z$. In this way, one
obtains~$(\XX_{\mathrm{L}}, \ZZ^{2})$ as a topological
dynamical system.

Let us now take the measure $\mu^{}_{\mathrm{L}}$ into account (a
remarkable feature of this system is that it has many invariant
measures -- see~\cite{E} for the details; we will only be interested
in the unique translation-invariant measure). The
space~$\XX_{\mathrm{L}}$ certainly contains some element~$v$ with
value~$-1$ at~$0$ (meaning that~$v^{}_{0}=-1$). The group structure
of~$\XX_{\mathrm{L}}$ then implies that
\[
   \mu^{}_{\mathrm{L}} \{ w \mid w^{}_{0} = 1 \} \, = \,
   \mu^{}_{\mathrm{L}} \bigl( v \cdot
   \{ w \mid w^{}_{0} = 1 \} \bigr) \, = \,
    \mu^{}_{\mathrm{L}} \{ w \mid w^{}_{0} = -1 \} \ts ,
\]
since $ \mu^{}_{\mathrm{L}}$ is the Haar measure, and hence
invariant under multiplication by the element $v$. On the other
hand, the whole group $\XX_{\mathrm{L}}$ is the disjoint union
of the two cylinder sets $\{ w \mid w^{}_{0} = 1 \}$ and
$\{ w \mid w^{}_{0} = - 1 \}$, so that
\begin{equation} \label{eq:L-onepoint}
   \mu^{}_{\mathrm{L}} \{ w \mid w^{}_{0} = 1 \} \, = \,
   \mu^{}_{\mathrm{L}} \{ w \mid w^{}_{0} = -1 \} \, = \,
   \frac{1}{2} \ts .
\end{equation}
Similarly, one finds for the cylinder sets defined by the
values~$\varepsilon$ at~$0$ and~$\varepsilon'$ at~$x\ne 0$ that
\begin{equation} \label{eq:L-twopoint}
     \mu^{}_{\mathrm{L}} \bigl\{ w \mid w^{}_{0} = \varepsilon
     \ts , \, w^{}_{x} = \varepsilon'  \ts \bigr\}
     \, = \,  \frac{1}{4}\ts ,
\end{equation}
for any choice of $\varepsilon,\varepsilon' \in \{\pm 1 \}$. To see
this, one observes that $\XX_{\mathrm{L}}$ certainly contains some
element $u$ that takes the values $1$ at $0$ and $-1$ at $x$, and
another, $v$ say, with the values at $0$ and $x$ interchanged. Using
the group structure of~$\XX_{\mathrm{L}}$ again, one can conclude that
the four cylinder sets given by the choices for~$\varepsilon$
and~$\varepsilon'$ must have the same measure.

The shift action $\alpha$ preserves the Haar measure
$\mu_{\mathrm{L}}$, and the~$\ZZ^2$ measure-preserving
dynamical system~$(\XX_{\mathrm{L}},\mu_{\mathrm{L}})$ was
introduced by Ledrappier \cite{L} to show that the (in general
still open) Rokhlin problem, which asks if a mixing
measure-preserving~$\ZZ$-action must be mixing of all orders,
has a negative answer for~$\ZZ^2$. A discussion of this system
in the context of Gibbs measures and extremality can be found
in~\cite{Sl}. The system is easily shown to be mixing, and to
have countable (dynamical) Lebesgue spectrum, but the following
argument shows that it is \emph{not} mixing of all orders.  The
relation~\eqref{eq:def-ledrap} propagates to show that
\[
     w^{}_{x}\, w^{}_{x+2^n \eone} w^{}_{x+2^n \etwo}=1
\]
holds for all~$n\ge 0$. Writing $A=\{w\mid w^{}_0=1\}$, it follows
that
\[
   \mu_{\mathrm{L}}  \left(  A\cap\alpha^{(-2^n,0)}
      A\cap\alpha^{(0,-2^n)}A \right)=\frac{1}{4}
\]
for all~$n\ge0$, so that $\alpha$ is not mixing on triples of
sets (or is not~$3$-mixing). Thus, as a two-dimensional system,
there are correlations between arbitrarily distant triples of
coordinates. In contrast to this, the one-dimensional
subsystems (that is, the measure-preserving
transformation~$\alpha^z$ for any fixed~$z\in\ZZ^2$) are all as
chaotic as possible:\ each map~$\alpha^z$ is measurably
isomorphic to a Bernoulli shift (that is, to an i.i.d.\
process).  This contrast between the properties of lower
dimensional subsystems and the whole system is a characteristic
feature of symbolic dynamics in higher dimensions.

\section{Autocorrelation and diffraction of ergodic subshifts}

Since all initial measures that appear in this article are
supported on the (planar) integer lattice, we formulate the
main concepts and results for this (slightly simpler) case; for
the general theory, we refer to~\cite{BM,Hof,Martin}. Let~$\XX
\subset \{ \pm 1\}^{\ZZ^{2}}$ be a subshift (meaning a
shift-invariant subset that is closed in the product topology)
with an ergodic invariant probability measure $\mu$ (more
generally, we also consider any closed, shift-invariant subset
of $(\SSS^{1})^{\ZZ^{2}}$, where $\SSS^{1} = \{z\in \CC \mid
\lvert z \rvert = 1 \}$).  Associate to each element $w\in\XX$
a \emph{Dirac comb} via
\begin{equation} \label{eq:def-comb}
    \omega \, = \, \omega_{w} \, := \sum_{x\in\ZZ^{2}}
              \! w_x \ts \delta_x \ts ,
\end{equation}
which is a translation bounded (and hence locally finite) measure on
$\ZZ^{2}$. At the same time, it can be considered as a measure on
$\RR^{2}$ with support in $\ZZ^{2}$. We will use both pictures in
parallel below, where measures (and their limits) are always looked at
in the vague topology. We will formulate the diffraction spectrum
starting from individual Dirac combs attached to elements of the shift
space, rather than via the invariant measure on $\XX$. This is
motivated by the physical process of diffraction, which deals with a
realisation of the process and not with the entire ensemble. For
ergodic systems (such as those considered here), this makes (almost)
no difference due to the ergodic theorem.

Let~$C_{N}=[-N,N]^2$ be the closed, centred square of side
length $2N$, and denote the restriction of $\omega$ to
$C^{}_{N}$ by $\omega^{}_{N}$.  The \emph{natural
autocorrelation measure}, or autocorrelation for short, is
defined as the vague limit
\begin{equation} \label{eq:def-auto}
     \gamma \, = \, \gamma^{}_{\omega} \, = \,
     \omega \circledast \widetilde{\omega}
      \,  :=  \lim_{N\to\infty}  \nts \frac{\omega^{}_{N} \nts *
      \widetilde{\omega^{}_{N}}}{\mathrm{vol} (C_N)}\ts ,
\end{equation}
provided that the limit exists, where $\widetilde{\nu}$ is the
`flipped-over' measure defined by $\widetilde{\nu} (g) = \overline{\nu
(\widetilde{g}\ts)}$ for arbitrary continuous functions $g$ of compact
support, and $\widetilde{g} (x) = \overline{g(-x)}$. Here and below,
we use $\circledast$ for the volume weighted (or Eberlein)
convolution.

When $\omega$ is a translation bounded measure, there is always at
least one accumulation point, by \cite[Prop.~2.2]{Hof}. Much of what
we say below remains valid for each such accumulation point
individually. Here, we will only consider situations where the limit
exists, at least almost surely in the probabilistic sense. If so, the
autocorrelation is, by construction, a positive definite measure, and
hence Fourier transformable \cite[Sec.~I.4]{BF}. The result, the
\emph{diffraction measure} $\widehat{\gamma}$, is a translation
bounded, positive measure on $\RR^{2}$ (by the Bochner-Schwartz theorem
\cite[Thm.~4.7]{BF}) with a unique Lebesgue decomposition
\begin{equation} \label{eq:decomp}
    \widehat{\gamma} \, = \,
    \bigl( \widehat{\gamma}\bigr)_{\mathsf{pp}} + \ts
    \bigl( \widehat{\gamma}\bigr)_{\mathsf{sc}} + \ts
    \bigl( \widehat{\gamma}\bigr)_{\mathsf{ac}}
\end{equation}
into its pure point, singular continuous and absolutely continuous
parts.  The measure $\widehat{\gamma}$ describes the outcome of
kinematic diffraction in crystallography and materials science
\cite{Cowley,Hof}. The pure point and absolutely continuous parts are
often referred to as Bragg spectrum and diffuse scattering,
respectively.

It is perhaps interesting to note that also singular continuous
components show up quite frequently, see \cite{Wit} for an overview,
and indicate subtle aspects of long-range order that are far from
being well understood.  The mathematical paradigm in one spatial
dimension is the classic Thue-Morse system, which was first analysed
by Wiener and Mahler, see \cite{HB,BG-TM} and references therein for
details.

However, $\widehat{\gamma}$ is also an interesting mathematical object
in its own right. It has only recently been properly appreciated in
the mathematical community, due to its connection with dynamical
systems theory; see \cite{Hof,Martin,BL, BLM} and references therein for
more.

Let us return to the autocorrelation $\gamma$ of a measure $\omega$
that corresponds to a point $w\in \XX_{\mathrm{L}}$ as
in~\eqref{eq:def-comb}. If it exists, a simple calculation shows that
it must be a pure point measure of the form
\begin{equation} \label{eq:auto-form}
    \gamma \, = \sum_{z\in \ZZ^{2}} \nts\nts  \eta(z) \ts \delta_{z}\ts ,
\end{equation}
where (with the cube $C_n$ from above) the \emph{autocorrelation
  coefficients} $\eta (z)$ are given by
\begin{equation} \label{eq:def-eta}
    \eta(z) \,  = \lim_{N\to\infty}
    \frac{1}{\lvert \ZZ^2 \cap C_{N} \rvert} \!
       \sum_{x\in \ZZ^{2}\cap C_{N}}
     \! \! \! \overline{w_{x}} \ts w_{x+z} \ts ,
\end{equation}
with $\lvert A \rvert$ denoting the cardinality of a (finite) set $A$.
When one considers $\eta$ as a function on $\ZZ^{2}$, it is again
positive definite.  There are several distinct ways to write $\eta(z)$
as a limit, which differ by `boundary terms' that vanish in the limit
as $N\to\infty$. We have chosen the most convenient for our purposes;
see the detailed discussion in \cite{Hof, Martin} for more on
this. For lattice systems, it is clear that the existence of the limit
in \eqref{eq:def-auto} is equivalent to the (simultaneous) existence
of the limits in \eqref{eq:def-eta} for all $z\in\ZZ^{2}$. We will use
this equivalence many times below without further notice; see
\cite[Thm.~4]{Lenz} for a direct formulation in a more general
setting.

A closer look at \eqref{eq:def-eta} reveals that $\eta(z)$ (which depends
on $w$) can be seen
as the orbit average of the function $h_{z} \! : \, \XX
\longrightarrow \CC$ defined by $h_{z} (w) = \overline{w^{}_{0}} \ts
w^{}_{z}$, which is continuous (and hence measurable) on
$\XX$.  In the ergodic case, we thus have
\begin{equation} \label{eq:eta-ergodic}
   \eta(z) \, = \, \mu (h_{z} ) \, = \int_{\XX} h_{z} (v)
   \dd \mu (v) \ts , \qquad \mbox{for a.e.\ $w\in\XX$,}
\end{equation}
by an application of Birkhoff's ergodic theorem; see \cite{W} or
\cite[Thm.~2.1.5]{K} for a formulation of the case at hand (with
$\ZZ^{2}$-action). Since $\ZZ^{2}$ is countable, we get the almost
sure existence of all coefficients $\eta(z)$, and hence the almost
sure existence of the autocorrelation $\gamma$ in this case.

Before we continue, let us briefly explain how to treat Dirac combs for
ergodic shifts with more general weights. When we use the (possibly
complex) numbers $\phi_{\pm}$ instead of $\pm 1$, the new Dirac comb
attached to $w\in\XX$ is
\[
     \omega^{}_{\phi} =
     \frac{\phi_{+} + \phi_{-}}{2}\, \delta^{}_{\ZZ^{2}} +
     \frac{\phi_{+} - \phi_{-}}{2}\, \omega \ts ,
\]
with $\omega$ as in \eqref{eq:def-comb}. Let us assume that the
original weights are balanced, in the sense that the cylinder sets
$\{w\mid w^{}_{0}=1 \}$ and $\{w\mid w^{}_{0}= -1 \}$ have the same
measure (as is the case for Ledrappier's shift). Ergodicity then
tells us that, for almost all $w\in \XX$,
\[
    \omega \circledast \delta^{}_{\ZZ^{2}} \, = \,
    \delta^{}_{\ZZ^{2}} \circledast \widetilde{\omega}
    \, = \, 0 \ts ,
\]
so that the new autocorrelation simply becomes
\[
    \gamma^{}_{\phi} =
    \frac{\lvert \phi_{+} + \phi_{-}\rvert^{2}}{4}\,
     \delta^{}_{\ZZ^{2}} +
     \frac{\lvert \phi_{+} - \phi_{-}\rvert^{2}}{4}\,
     \gamma \ts ,
\]
where $\gamma$ is the autocorrelation of $\omega$. Recalling
that Poisson's summation formula for lattice Dirac combs
(compare \cite{BM} and references given there) gives the
self-dual formula
\begin{equation} \label{eq:PSF}
   \widehat{\delta^{}_{\ZZ^{2}}} = \delta^{}_{\ZZ^{2}}   \ts ,
\end{equation}
one finds the new diffraction measure as
\begin{equation} \label{eq:gen-diffraction}
   \widehat{\gamma^{}_{\phi}} \, = \,
   \frac{\lvert \phi_{+} + \phi_{-}\rvert^{2}}{4}\,
     \delta^{}_{\ZZ^{2}} +
     \frac{\lvert \phi_{+} - \phi_{-}\rvert^{2}}{4}\,
     \widehat{\gamma} \ts ,
\end{equation}
which is thus a mixture of a pure point measure with whatever
$\widehat{\gamma}$ is. Note that the point measure
$\delta^{}_{\ZZ^{2}}$ is trivial (that is, carries no information
about the system) -- it only reflects the fact that the support of
$\omega^{}_{\phi}$ is the integer lattice. This corresponds to the
suspension of the dynamical system $(\XX_{\mathrm{L}},\ZZ^{2})$
into its continuous counterpart under the action of the group $\RR^{2}$.

This observation is the reason why we prefer to work with balanced
weights, and hence the motivation for our preference of the
multiplicative formulation chosen above over the usual additive one.
It is clear how to generalise this to other lattices and other types
of shifts.

\section{Diffraction of Bernoulli combs}

Here, we briefly summarise the known results on the autocorrelation
and diffraction of multi-dimensional Bernoulli systems, which includes
(fair) coin tossing as a special case.  Let $(W_{x})^{}_{x\in\ZZ^{2}}$
be a family of i.i.d.\ random variables with values in $\SSS^{1}$,
common law $\nu$ and representing random variable $W$. This defines a
Bernoulli shift on $\ZZ^{2}$, which is known to be ergodic (in fact,
mixing of all orders).

Consider now the Dirac
comb $\sum_{x\in\ZZ^{2}} W_{\nts x}\ts \delta_{x}$, which is a random
measure on $\ZZ^{2}$, or on $\RR^{2}$ with support $\ZZ^{2}$.  The
autocorrelation coefficients are given by $\eta(0)=1$ together with
$\eta (z) = \bigl| \EE_{\nu} (W) \bigr|^{2}$ (a.s.)  for all $z\ne
0$. This result can also be derived by an application of the strong
law of large numbers (SLLN), which permits a significant
generalisation to non-lattice systems; see \cite{BBM} and references
therein for more.  Also, the weaker notion of pairwise independence is
then sufficient, due to Etemadi's version \cite{Ete} of the SLLN.  A
systematic way to write the result is
\begin{equation} \label{eq:coin-auto}
   \eta (z) \, = \,  \bigl| \EE_{\nu} (W) \bigr|^{2}  +
   \bigl( \EE_{\nu} \bigl( \lvert W \rvert ^{2}\bigr) -
    \bigl| \EE_{\nu} (W) \bigr|^{2} \bigr) \delta^{}_{z,0}
   \qquad \mbox{(a.s.)}.
\end{equation}
This formulation remains valid even if we allow for more general
random variables. The following result is now straight-forward.

\begin{theorem} \label{thm:Bernoulli}
  Let $(W_{\nts x})^{}_{x\in\ZZ^{2}}$ be a family of i.i.d.\ random
  variables with values in\/ $\SSS^{1}$ and common law $\nu$. The
  associated random Dirac comb $\omega = \sum_{x\in\ZZ^{2}} W_{\nts x}\ts
  \delta_{x}$ almost surely has the autocorrelation and diffraction
  measures
\[
   \gamma \, = \, \bigl| \EE_{\nu} (W) \bigr|^{2} \, \delta^{}_{\ZZ^{2}}
       + \mathrm{cov} (W) \ts \delta^{}_{0} \quad \mbox{and} \quad
   \widehat{\gamma} \, = \, \bigl| \EE_{\nu} (W) \bigr|^{2} \,
       \delta^{}_{\ZZ^{2}}  + \mathrm{cov} (W)\ts \lambda \ts ,
\]
   where $\lambda$ is Lebesgue measure and
   $\mathrm{cov} (W) :=\, \EE_{\nu} \bigl( \lvert W \rvert ^{2}\bigr) -
   \bigl| \EE_{\nu} (W) \bigr|^{2}$ the covariance of $W$.
\end{theorem}
\begin{proof}
  The structure of the (countably many) autocorrelation coefficients
  explained above gives the almost sure form of $\gamma$ by a simple
  calculation. Its Fourier transform exists, and has the claimed form
  due to $\widehat{\delta^{}_{0}} =\lambda$ and the Poisson summation
  formula for lattice Dirac combs \eqref{eq:PSF} mentioned above.
\end{proof}

Theorem~\ref{thm:Bernoulli} has an interesting special case that
includes the (fair) coin tossing sequence and makes use of the
fact that $\EE_{\nu} \bigl( \lvert W \rvert^{2}\bigr) = 1$ for
random variables with values in $\SSS^{1}$.
\begin{coro}
  Consider the situation of Theorem~$\ref{thm:Bernoulli}$ under the
  additional assumption that\/ $\EE_{\nu} (W) = 0$. Then, the
  autocorrelation and diffraction measures simplify to
\[
   \gamma \, = \, \delta^{}_{0} \quad \mbox{and} \quad
   \widehat{\gamma} \, = \, \lambda \ts ,
\]
  again almost surely as before.  \qed
\end{coro}

The underlying two-dimensional Bernoulli shift has positive entropy,
for instance $\log(2)$ when $W$ takes values $\pm 1$ with equal
probability. It has long been known, due to work by Rudin and Shapiro
\cite{R, S}, that one can also construct a deterministic sequence
(with zero entropy) with the same autocorrelation and diffraction
\cite{HB}. Though the orignal construction was for $\ZZ$, it has an
immediate generalisation to $\ZZ^{2}$ (and to $\ZZ^{d}$ as well), see
\cite{HB} for details.  Moreover, one can modify the system (by
`Bernoullisation' \cite{BG}) to construct an isospectral transition
from the deterministic case to the Bernoulli system with continuously
varying entropy \cite{BG}.  This does not yet explore the full
scenario of $\ZZ^{2}$ shifts though. In particular, we now have the
possibility of genuinely two-dimensional systems (by which we mean to
exclude simple Cartesian products) with entropy of rank $1$.

\section{Autocorrelation and diffraction of Ledrappier's shift}

Consider the measure dynamical system $(\XX^{}_{\mathrm{L}}, \ZZ^{2},
\mu^{}_{\mathrm{L}})$ as introduced in Section~\ref{sec:Ledrappier}.
To any $w\in\XX^{}_{\mathrm{L}}$, we attach the Dirac comb as defined in
Eq.~\eqref{eq:def-comb} and consider its autocorrelation.
\begin{lemma} \label{lem:L-coeff}
   Let $w\in \XX_{\mathrm{L}}$ and $\omega$ as in \eqref{eq:def-comb}.
   The corresponding autocorrelation coefficients satisfy
   $\eta^{}_{\ts\mathrm{L}} (0) = 1$ and, $\mu^{}_{\mathrm{L}}$-almost
   surely, $\eta^{}_{\ts\mathrm{L}} (z) = 0$ for all\/ $0\ne z\in\ZZ^{2}$.
\end{lemma}
\begin{proof}
  The identity $\eta^{}_{\ts\mathrm{L}} (0) = 1$ clearly holds for all
  $w\in \XX_{\mathrm{L}}$. Let now $z\in\ZZ^{2}\setminus \{0\}$ be
  fixed. Ledrappier's shift is ergodic (indeed, is mixing; see
  \cite{L} or \cite{Sch1} for details), so that we can employ
  Birkhoff's theorem as outlined above. This gives,
  $\mu^{}_{\mathrm{L}}$-almost surely,
\[
   \eta^{}_{\ts\mathrm{L}} (z) \, = \, \mu^{}_{\mathrm{L}} (h_z) \, =
   \!\! \sum_{\varepsilon, \varepsilon' \in \{\pm 1\}}\!\!
   \varepsilon \ts \varepsilon{\ts}' \! \cdot
   \mu^{}_{\mathrm{L}} \bigl\{ w\in\XX_{\mathrm{L}} \mid
   w^{}_{0} = \varepsilon \ts , \, w^{}_{-z} = \varepsilon' \ts
   \bigr\} \, = \, 0 \ts ,
\]
because all cylinder sets under the sum have equal measure $1/4$ by
Eq.~\eqref{eq:L-twopoint}. Since there are countably many such
coefficients, the claim follows.
\end{proof}

A simple calculation now gives the following result.
\begin{theorem} \label{thm:L-diffract}
  Let $w$ be any element of the Ledrappier shift, in its
  multiplicative formulation used above, and let $\omega$ be the
  attached Dirac comb of \eqref{eq:def-comb}. Then,
  $\mu^{}_{\mathrm{L}}\nts $-almost surely,
\[
      \gamma \, = \, \delta^{}_{0}   \quad \mbox{and} \quad
      \widehat{\gamma} \, = \, \lambda
\]
   are the corresponding autocorrelation and diffraction measures.
   \qed
\end{theorem}

This means that the Bernoulli shift and the Ledrappier shift are
homometric, despite the fact that the latter has zero two-dimensional
entropy with complete correlations between coordinates $0$,~$2^n
\eone$ and~$2^n \etwo$ for any~$n\ge 0$ as explained earlier. The
Bernoulli shift has vanishing $3$-point correlation here, and thus
differs.

As mentioned above, there are \emph{many} other invariant
measures on~$\XX_{\mathrm{L}}$ by~\cite{E}, including ones with
any given entropy for $\alpha^{\eone}$ (it is not entirely
clear what further entropy or correlation conditions would
constrain the possible invariant measures to be Haar measure;
Schmidt~\cite{Sch2} has partial results relating the absence of
higher-order correlations to measure rigidity for systems
like~$\XX_{\mathrm{L}}$).

Also, $\mu^{}_{\mathrm{L}}$ can be viewed as an ergodic measure on
the full shift that is concentrated on $\XX_{\mathrm{L}}$. As such, it is
mutually singular with the Bernoulli (or i.i.d.) measure.

\begin{remark}[Extension of Theorem~\ref{thm:L-diffract}]
  A similar argument applies to a whole class of systems which may be
  constructed as follows. Take any prime ideal~$P\subset R_d=\ZZ
  [x_1^{\pm1},\dots,x_d^{\pm1}]$ with the following properties:
\begin{itemize}
\item $P$ contains a prime~$p$ (this makes the system a closed
  subshift of finite type inside the full~$d$-dimensional $p$-shift,
  and in particular disconnected);
\item $P\cap\{x_1^{a_1}\cdots
    x_d^{a_d}-1\mid(a_1,\dots,a_d)\in \ZZ^{d}\}=\{0\}$, so
    the corresponding system is mixing;
\item the Krull dimension\footnote{This is calculated as $(d+1)$ minus
    the number of independent generators of~$P$. If the Krull
    dimension is~$d$, the system is the full~$d$-dimensional i.i.d.\
    shift on~$p$ symbols, since $P=pR_d$ is then the only possibility;
    as~$k$ decreases, the entropy rank gets smaller until one gets
    down to examples like Ledrappier's shift with~$k=1$. There, each
    single element is measurably isomorphic to a finite-entropy
    i.i.d.\ process.}  of~$R_d/P$ is~$k$ with~$1\le k\le d$.
\end{itemize}
Then, a similar conclusion holds for generic points in the
corresponding~$\ZZ^d$-action. Later, we will also see an extension to
connected examples (with continuous local degree of freedom).
\end{remark}

\section{Correlation functions}

So far, we have concentrated on the autocorrelation measure,
because its knowledge immediately gives the diffraction measure
via the Fourier transform. The coefficients $\eta(m)$ can be seen
as volume averaged two-point correlations, and it is natural
to also consider other correlation functions. To this end, we
define the $n$-point correlation
\begin{equation} \label{eq:def-corr}
   \big\langle z^{}_{1}, \ldots , z^{}_{n} \big\rangle (w)
   \, := \lim_{N\to\infty} \frac{1}
         {\lvert \ZZ^{2} \cap C_{N} \rvert}
   \sum_{x\in \ZZ^2\cap C_{N}}w^{}_{x+z_{1}}\!  \cdot \ldots
   \cdot w^{}_{x+z_{n}}
\end{equation}
for any fixed $w\in \XX$ and $n\in\NN$, where all $z_i \in \ZZ^{2}$.
We only consider situations where the limit exists, which is
(a.s.) the case when $\XX$ is an ergodic subshift. Then, the
ergodic theorem implies that
\begin{equation} \label{eq:corr-mean}
   \big\langle z^{}_{1}, \ldots , z^{}_{n} \big\rangle (w) \, = \,
   \int_{\XX}  v^{}_{z_{1}}\! \cdot \ldots \cdot
    v^{}_{z_{n}} \dd \mu (v)
\end{equation}
holds for almost every $w$, which also explains why we
simply write $\langle z^{}_{1}, \ldots , z^{}_{n} \rangle$
from now on. It is clear that the $n$-point correlations
are translation invariant in the sense that
\[
    \langle z^{}_{1}, \ldots , z^{}_{n} \rangle  \, = \,
    \langle x+z^{}_{1}, \ldots , x+z^{}_{n} \rangle
\]
holds for all~$x\in\ZZ^{2}$.

\begin{remark}
In this new notation, we find $\eta(z) = \langle 0, z \rangle$ for
$z\in\ZZ^{2}$, both for the Bernoulli (B) and the Ledrappier (L)
shift, because all weights are real. On the other hand, we find
$\langle z^{}_{1},z^{}_{2} \rangle^{}_{\mathrm{B}} = 0$ for any
$z^{}_{1} \ne z^{}_{2}$, hence also $\langle 0, z^{}_{1}, z^{}_{2}
\rangle^{}_{\mathrm{B}} =0$, while we earlier observed that $\langle
0, 2^n \eone, 2^n \etwo \rangle^{}_{\mathrm{L}} = 1$ for all $n\ge 0$,
which expresses the difference between the two shifts on the level of
three-point correlations.
\end{remark}

Since we also want to consider $\SSS^{1}$ as the local degree of
freedom, we need complex correlation functions as well. Here, we write
$z^{*}_{i}$ on the lefthand side of \eqref{eq:def-corr} to indicate
the factor $\overline{w^{}_{x+z_{i}}}$ on the right. The almost sure
validity of \eqref{eq:corr-mean} extends to this more general setting
in an obvious way.
Now, the relation between autocorrelation coefficients and $2$-point
correlations is given by
\[
    \eta(z) \, = \, \langle 0^{*}, z \rangle
    \, = \, \langle 0, - z ^{*} \rangle \ts ,
\]
where we have used the translation invariance. The
following result is immediate.

\begin{lemma} \label{lem:zero-corr}
  Let $(W_{\nts x})_{x\in\ZZ^{2}}$ be an i.i.d.\ family of\/ random
  variables, repesented by\/ $V$ with values in\/ $\SSS^{1}$ and law
  $\,\nu$.  Assume that\/ $V$ has zero mean, $\EE_{\nu} (V) =0$.  When
  $z^{}_{1}, \ldots , z^{}_{n} \in \ZZ^{2}$ are distinct, one almost
  surely has
\[
    \langle z^{}_{1}, \ldots , z^{}_{n} \rangle \, = \, 0\ts .
\]
   This relation also holds for the complex correlations, as
   long as the $z_{i}$ are distinct.
\end{lemma}
\begin{proof}
The corresponding Bernoulli system is mixing, hence ergodic.
Eq.~\eqref{eq:corr-mean} results in a factorisation of the
righthand side. One obtains $n$ factors,
each of which (due to the i.i.d.\ nature of
the $W_{\! z_{i}}$) is an integral of the form
\[
       \int_{\SSS^{1}} v \dd \nu (v)
       = \EE_{\nu} (V) = 0 \ts ,
\]
so that the entire correlation function also vanishes.

In the case of complex correlations, as long as the $z_{i}$
are distinct, each factor is either of this form or its
complex conjugate, with the same conclusion.
\end{proof}

It is clear how one can generalise this result to also include powers,
where one then needs further assumptions on higher moments of the law
$\nu$.  Let us state one version that will become relevant below. To
do so, we need a slightly more general notation for the mixture of
moments and correlations. We write $\langle \ldots , (z_i, m_i) ,
\ldots \rangle$ for the correlation function where we take the $i$th
factor as $W_{x+z_i}^{m_i}$ in the orbit average, with $m_i \in \ZZ$
arbitrary.  Since we work on $\SSS^{1}$, this includes complex
conjugation via $m_i = -1$.

\begin{prop}  \label{prop:corr-bern}
   Consider the full shift on $\ts \XX=(\SSS^{1})^{\ZZ^{2}}\!$, with the
   invariant measure $\mu$ that is constructed from the uniform
   distribution on\/ $\SSS^{1}$ via the Bernoulli $($or i.i.d.$)$
   process. Then, the generalised correlations $\big\langle (z^{}_1,
   m^{}_1), \ldots , (z^{}_{n}, m^{}_{n}) \big\rangle$ with $n\in\NN$
   and distinct $z_i\in\ZZ^2$ almost surely exist. In particular, they
   all vanish, unless $m^{}_{1} = \ldots =m^{}_{n} =0 $, when they are
   $1$.
\end{prop}
\begin{proof}
The Bernoulli measure $\mu$ is ergodic, so that the almost sure
existence claim is clear. By the ergodic theorem together with the
i.i.d.\ property, one finds
\[
     \big\langle (z^{}_1, m^{}_1), \ldots ,
            (z^{}_{n}, m^{}_{n}) \big\rangle
     = \prod_{j=1}^{n} \int_{\XX} w_{z_j}^{m_j} \dd \mu (w)
     = \prod_{j=1}^{n} \int_{0}^{1} e^{2 \pi i \ts m_j y} \dd y \ts ,
\]
where the last step is just one way to write down the uniform
distribution at (any) one site. Our claim is now obvious.
\end{proof}

\section{The $(\times 2,\times 3)$-shift}

A related example, with continuous site space $\SSS^{1}$, is the
following. Define
\[
   \XX_{\mathrm{F}}\, = \,
   \big\{w\in(\SSS^{1})^{\ZZ^2}\mid w_{z+\eone}=w_{z}^2, \,
   w_{z+\etwo}=w_{z}^3 \, \mbox{ for all }z\in\ZZ^{2} \big\}
\]
(that is, the $(\times 2,\times 3)$-shift \cite{Ru} written
multiplicatively, with ergodic invariant measure
$\mu^{}_{\mathrm{F}}$).  Then, we may consider the complex-valued
Dirac comb
\begin{equation} \label{eq:def-comb-23}
    \omega \, = \, \omega_{w}
    \, = \sum_{x \in\ZZ^2} w_{x} \ts \delta_{x} \ts .
\end{equation}
As before, using the ergodic theorem, we see that the autocorrelation
coefficients $\eta(z) = \langle 0^{*},z\rangle$ almost surely exist,
and that their calculation reduces to computing an integral over
$\XX_{\mathrm{F}}$.

\begin{theorem} \label{thm:23}
 Let $w$ be any element of the $(\times 2,\times 3)$-shift, in its
  multiplicative formulation used above, and let $\omega$ be the
  attached Dirac comb of \eqref{eq:def-comb-23}. Then,
  $\mu^{}_{\mathrm{F}}\nts $-almost surely,
\[
      \gamma \, = \, \delta^{}_{0}   \quad \mbox{and} \quad
      \widehat{\gamma} \, = \, \lambda
\]
   are the corresponding autocorrelation and diffraction measures.
\end{theorem}
\begin{proof}
The almost sure existence of $\gamma$ follows from that of the
countably many coefficients $\eta(z)$ with $z\in\ZZ^2$. To show
the main claim, we need to show the almost sure validity of
$\eta(z) = \delta_{z,0}$. Since $\eta(-z) = \overline{\eta(z)}$
by the positive definiteness of $\eta$ as a function on $\ZZ^2$,
it suffices to consider $z=m\eone + n\etwo$ with $m\ge 0$ and
$n\in\ZZ$.

Observe that we have
\[
    \eta(z)\, = \lim_{N\to\infty} \frac{1}{\lvert
    \ZZ^{2} \cap C_{N} \rvert }
    \sum_{x\in \ZZ^2 \cap C_{N}}\! \overline{w_{x}} \, w_{x+z} \ts ,
\]
and consider the case with $n\ge 0$ first. The rule for the
translation action on $\XX_{\mathrm{F}}$ implies that
\[
      w_{x+z} = \bigl( w_{x} \bigr)^{2^m \ts 3^n},
\]
so that $\eta(z)$ can be viewed (as $w_{x} \in \SSS^{1}$) as an orbit
average of $(w_{x})^{2^m\ts 3^n - 1}$, with the almost sure limit
\[
   \int_{\XX_{\mathrm{F}}}\! \bigl(w^{}_{0}\bigr)^{2^m \ts 3^n - 1}
         \dd \mu^{}_{\mathrm{F}} (w) \, =
   \int_{0}^{1} \exp\bigl(2 \pi i\ts (2^m \ts 3^n -1) x\bigr) \dd x
    = \begin{cases} 1 , & \text{if $m=n=0$}, \\
                    0 , & \text{otherwise},
      \end{cases}
\]
where $m=n=0$ means $z=0$. Here, the first equality uses the fact that
the invariant measure $\mu^{}_{\mathrm{F}}$ reduces to the uniform
distribution on $\SSS^{1}$ when marginalised on all sites but one.

Similary, when $m\ge 0$ with $n<0$, we set $y=x-\lvert n\rvert \etwo$,
so that $\overline{w_{x}} = (\overline{w_{y}})^{3_{}^{| n |}}$
and $w_{x+z} = (w_{y})^{2^m}$, whence $\eta(z)$ is now an orbit
average of $(w_{y})^{2^m - 3^{\lvert n \rvert}}$, hence almost
surely given by the integral
\[
   \int_{0}^{1} \exp\bigl(2 \pi i\ts (2^m - 3^{\lvert n \rvert})
         x \bigr) \dd x
\]
which always vanishes, as $n<0$ was assumed.

Putting the arguments together indeed (almost surely) results in
$\eta(z) = \delta_{z,0}$, hence $\gamma = \delta_0$ and
$\widehat{\gamma}=\lambda$ as claimed.
\end{proof}

Quite remarkably, the similarities between the Bernoulli shift
$(\SSS^{1})^{\ZZ^{2}}$ and the $(\times 2, \times 3)$-shift
go a lot further.

\begin{theorem} \label{thm:general-23}
  Consider the $(\times 2,\times 3)$-shift in its multiplicative
  version as used above, with invariant measure
  $\mu^{}_{\mathrm{F}}$. If $z^{}_{1}, \ldots , z^{}_{n}$ are $n$
  arbitrary points of the lattice\/ $\ZZ^2$, we almost surely have
  $\langle z^{}_1, \ldots , z^{}_{n} \rangle^{}_{\mathrm{F}} = 0$.
  More generally, with $z^{}_{j} = k^{}_{j} \eone +
  \ell^{}_{j} \etwo$, one almost surely obtains
\[
   \big\langle (z^{}_1, m^{}_1), \ldots , (z^{}_{n}, m^{}_{n})
   \big\rangle^{}_{\mathrm{F}}
   = \begin{cases}
      1 ,  & \text{if $\sum_{j=1}^{n} m_{j}
                   \ts 2^{k_{j}} \ts 3^{\ell_{j}} = 0$}, \\
      0,  & \text{otherwise},
   \end{cases}
\]
  for the generalised correlation coefficients.
\end{theorem}
\begin{proof}
To calculate the generalised correlation coefficients,
we need a variant of the trick used in Theorem~\ref{thm:23}.
Given $z^{}_{1}, ... , z^{}_{n}$ with coordinates as
stated, define $y = k \eone + \ell \etwo$ with $k=\min_{i} k_{i}$
and $\ell = \min_{i} \ell_{i}$, so that $y$ is the lower left corner
of the smallest lattice square (with edges parallel to $\eone$ and
$\etwo$) that contains all the $z_{i}$. One can now check that
\[
   \begin{split}
    \big \langle (z^{}_{1}, & m^{}_{1}),  \ldots ,
       (z^{}_{n}, m^{}_{n}) \big \rangle^{}_{\mathrm{F}} \, =
    \lim_{N\to\infty} \frac{1}{\vert \ZZ^{2} \cap C_{N}
        \rvert } \sum_{x\in \ZZ^{2} \cap C_{N}}
    \prod_{j=1}^{n} \bigl( w_{x+y} \bigr)^{m_{j} \ts 2^{k_{j} - k}
        \ts 3^{\ell_{j} - \ell}} \\ & =
    \int_{\XX^{}_{\mathrm{F}}}
       \bigl( w_{y} \bigr)^{\sum_{j=1}^{n} m_{j} \ts 2^{k_{j} - k}
        \ts 3^{\ell_{j} - \ell}} \dd \mu^{}_{\mathrm{F}} (w) \, =
    \int_{0}^{1} e^{2 \pi i \ts \sum_{j=1}^{n} m_{j}
       ( 2^{k_{j} - k} \ts 3^{\ell_{j} - \ell}) \ts t} \dd t \ts ,
   \end{split}
\]
where the second last step holds (almost surely) by the ergodic
theorem, while the last results from the appropriate marginalisation
of the measure $\mu^{}_{\mathrm{F}}$, as before.

{}For the first claim, all $m_i=1$, so that the exponent under the
integral is of the form $2\pi i \ts N t$ with $N$ a positive integer,
wherefore the integral vanishes. More generally, since
$N=\sum_{j=1}^{n} m_{j} ( 2^{k_{j} - k} \ts 3^{\ell_{j} - \ell})$ is
always an integer by construction, the integral vanishes unless
$N=0$, where it is $1$. This happens precisely when
$\sum_{j=1}^{n} m_{j}\ts 2^{k_{j}} \ts 3^{\ell_{j}} = 0$.
\end{proof}

\begin{remark}\label{remarkaboutcorrelationsintimes2times3}
In comparison, we have the somewhat surprising situation that the
full two-dimensional Bernoulli shift $(\SSS^1)^{\ZZ^{2}}$
and the $(\times 2, \times 3)$-shift are more or less indistinguishable
by correlation functions. There are differences though, such as the
relation
\[
   \big\langle \bigl((0,1),-2 \bigr), \bigl((1,1),1\bigr)
   \big\rangle_{\mathrm{F} } =1 \ts ,
\]
which rests on the identity $-2\cdot 2^{0} \ts 3^{1}+1\cdot
2^{1} \ts 3^{1}=0$, whereas the corresponding correlations for
functions determined by distinct single coordinates in the
Bernoulli shift always vanish. The condition $\sum_{j=1}^{n}
m_{j} 2^{k_j} 3^{\ell_j}=0$ appearing in
Theorem~\ref{thm:general-23} corresponds to a correlation
between the characters (that is, functions) defined by the
pairs $(z_{i} , m_{i})$,  where $(z,m)$ defines the character
that takes the value~$\exp({2\pi i m w_{z}})$ on the
point~$w=(w_z)\in  \XX_{\mathrm{F}}$. The deep $S$-unit theorem
of Schlickewei~\cite{Schli} shows that this condition will only
be achieved in an essential way (that is, without a shorter
vanishing sum) finitely often, and this Diophantine result is
equivalent to the system being mixing of all orders
by~\cite{SW}.
\end{remark}

Clearly, the $(\times 2, \times 3)$-shift is a nullset for the
Bernoulli measure, but the support of the invariant measure is not
something that is usually known at the beginning of the inference
problem, so that this result is further `bad news' for the inverse
problem of structure determination.

\begin{remark}
The way in which correlations between characters (or between
trigonometric polynomials) on~$\XX_{\mathrm F}$ vanish may be
viewed as a form of multiple~$m$-dependence in the sense
of~\cite{AGKV}. For a given collection of trigonometric
polynomials, the determination of a minimal value of~$m$ is a
rather subtle Diophantine problem, but as indicated in
Remark~\ref{remarkaboutcorrelationsintimes2times3}, such an~$m$
will always exist.
\end{remark}

\begin{remark}
In contrast to the Ledrappier shift from Section~\ref{sec:Ledrappier}, the
$(\times 2,\times 3)$-system has few invariant measures in
the following sense. If $m$ is an ergodic invariant probability
measure that gives some $z\in\mathbb Z^2$ positive
entropy, then $m$ must be the Haar measure by a result
of Rudolph~\cite{Ru}.

The two systems~$\XX_{\mathrm{L}}$ and~$\XX_{\mathrm{F}}$ have
interesting overall properties as follows; in particular, in terms of
having invariant measures, $\XX_{\mathrm{L}}$ is closer to i.i.d.\
than~$\XX_{\mathrm{F}}$, while $\XX_F$ is closer in terms of mixing.
Some of the properties are summarised in Table~\ref{tab-1}.
\end{remark}

\begin{table}
\begin{tabular}{|c|c|c|c|c|}
\hline & Rudin-Shapiro
&Ledrappier&$(\times2,\times3)$&2-dim.\ i.i.d.\\
\hline
$\vphantom{\sum^A}$$\mathbb Z^2$ entropy&
$0$ & $0$&$0$&$>0$\\
\hline
entropy rank&
$0$ & $1$&$1$&$2$\\
\hline
$\vphantom{\sum^A}$
mixing of all orders?&
No & No & Yes &Yes\\
\hline
$\vphantom{\sum^A}$
many inv. measures?&
No & Yes &No &Yes\\
\hline
pure Lebesgue diffr.?&
Yes & Yes & Yes & Yes \\
\hline
\end{tabular}

\bigskip

\caption{Basic information and comments on some of the systems discussed.}
\label{tab-1}
\end{table}

\section{Further remarks and outlook}

Above, we have concentrated on systems that are supported on
$\ZZ^{2}$.  When one lifts this restriction, one naturally enters the
realm of point process theory. An interesting example there is
provided by the classic Poisson process and its many
siblings. Starting from a homogeneous and stationary Poisson process
of (point) density $1$, and turning it into a marked point process by
adding weights $\pm 1$ with equal probability to the points of any
given realisation, produces a randomly weighted point set with (almost
sure) autocorrelation $\delta^{}_{0}$ and diffraction $\lambda$; see
\cite{BBM} for a proof and further details and examples.

It is well-known that lattice dynamical systems of algebraic
origin are spectrally indistinguishable (at least
once they are mixing) in the sense that they all have countable
Lebesgue dynamical spectrum. Above, we saw that they also have
an absolutely continuous diffraction spectrum. While the connection
between dynamical and diffraction spectra is well understood for
the pure point case (see \cite{BL} and references therein), this is
not so in general (as follows from \cite{ME} by way of counterexample).
However, the existing
body of examples seems to indicate that for systems with pure
Lebesgue spectrum, perhaps under some mild additional condition,
there is again a closer connection between these two types of
spectra, and it would be nice to understand this better.

\smallskip
\section*{Acknowledgements}
It is a pleasure to thank Aernout van Enter, Robert V.\ Moody
and Anthony Quas for helpful discussions, and the reviewers for
several useful comments. This work was supported by the German
Research Council (DFG), within the CRC 701.

\end{document}